\documentclass[12pt, reqno]{amsart}
\usepackage[a4paper, margin=1in]{geometry}
\usepackage{amsmath, amssymb, amsthm}
\usepackage{microtype}
\usepackage[colorlinks=true, citecolor=blue, linkcolor=blue]{hyperref}

\newtheorem{theorem}{Theorem}[section]
\newtheorem{corollary}[theorem]{Corollary}
\newtheorem{lemma}[theorem]{Lemma}
\newtheorem{proposition}[theorem]{Proposition}
\theoremstyle{definition}
\newtheorem{definition}[theorem]{Definition}
\newtheorem{remark}[theorem]{Remark}
\newtheorem{example}[theorem]{Example}

\numberwithin{equation}{section}

\usepackage{tikz}

\begin{document}
	
	\title[Betti Numbers of SCM Co-Chordal Graphs]{Betti Numbers of Sequentially Cohen-Macaulay Co-Chordal Graphs and Their Applications}
	
	\author{Mohammed Rafiq Namiq}
	\address{Department of Mathematics, College of Science, University of Sulaimani, Kurdistan Region, Iraq}
	\email{mohammed.namiq@univsul.edu.iq}
	
	
	\begin{abstract}
We study edge ideals of sequentially Cohen--Macaulay co-chordal graphs through the maximal-clique structure of their chordal complements.  After making the required construction order explicit, we derive a closed formula for the graded Betti numbers of a co-chordal graph whose complement is a $(d_1,\ldots,d_e)$-tree, and we characterize the Cohen--Macaulay case by its Betti sequence.  The general formula is then specialized to split and threshold graphs, prime ideal graphs of finite rings, nilpotent graphs of products of finite chain rings, and zero-divisor graphs of finite chain rings.  For products of chain rings, we characterize when the nilpotent graph is threshold and compute its Betti numbers in that range.  Finally, among co-chordal zero-divisor graphs of $\mathbb Z_n$, we determine exactly when the quotient is sequentially Cohen--Macaulay: this occurs for $n=p^a$, $n=2p$, and $n=2p^2$, with $p$ odd in the last two cases.
\end{abstract}
	
	\maketitle
	
	\section{Introduction}

Let $G=(V,E)$ be a finite simple graph on $N$ vertices, let $S=\mathbb K[x_v:v\in V]$, and let $I(G)$ be its edge ideal.  We write $\Delta_G=\operatorname{Ind}(G)$ for the independence complex of $G$.  Equivalently, $\Delta_G$ is the clique complex of the complement $\overline G$, and $S/I(G)$ is its Stanley--Reisner ring.  We use $\beta_{i,j}(S/I(G))$ for the graded Betti numbers and $\beta_i(S/I(G))=\sum_j\beta_{i,j}(S/I(G))$ for the total Betti numbers.

Fr\"oberg's theorem states that $I(G)$ has a $2$-linear resolution if and only if $\overline G$ is chordal \cite{froeberg1990}.  Hence, if $G$ is co-chordal, then
\[
\beta_i(S/I(G))=\beta_{i,i+1}(S/I(G))\qquad(i\ge1).
\]
The resolutions of $I(G)$ and $S/I(G)$ are related by
\[
\beta_{i,j}(I(G))=\beta_{i+1,j}(S/I(G))
\qquad(i\ge0).
\]

The present work concerns the additional sequentially Cohen--Macaulay condition within the co-chordal class.  Ahmed, Mafi, and Namiq characterized these graphs by requiring the chordal complement to be a $(d_1,\ldots,d_q)$-tree \cite{ahmed2025}.  A related formulation uses stable connectivity of the clique complex \cite{namiq2026_icm}.  An alternating formula for the linear strand of successively glued clique complexes was obtained in \cite{namiq2026_explicit}.  These results suggest that the homological invariants should be recoverable directly from a ridge-compatible construction order of the maximal cliques of $\overline G$.

Our first objective is to make this reduction explicit.  Once the maximal cliques are placed in a nonincreasing ridge-gluing order, the alternating formula simplifies to a closed binomial expression depending only on their ordered sizes.  We use this expression to determine projective dimension and depth and to characterize the Cohen--Macaulay case by its Betti sequence.  The converse in the latter characterization is proved from the Hilbert series, without assuming termwise uniqueness of a binomial expansion.

We next apply the formula to split and threshold graphs.  Split graphs were introduced by F\"oldes and Hammer \cite{foldes1977}; threshold graphs form a more restrictive subclass characterized by nested neighborhoods and several equivalent structural properties \cite{chvatal1977,hammer_ibaraki1981,mahadev1995}.  Betti numbers of split graphs have been studied by Singh and Verma \cite{singh2020}, and Fr\"oberg also attempted to work on split graphs \cite{froeberg2026}.  Our derivation comes from the $(d_1,\ldots,d_q)$-tree structure, treats zero cross-degrees uniformly, and gives consequences for projective dimension and Cohen--Macaulayness.  We also record specializations to complete split graphs, single-edge-deleted complete split graphs, and pineapple graphs.  Related formulas for broader split--join constructions appear in \cite{rather2026_splitjoin}.

Graph constructions associated with commutative rings provide a second group of applications.  We consider prime ideal graphs \cite{salih2022,rasheed2026}, the induced zero-product graph on nonzero nilpotent elements, and zero-divisor graphs in the sense of Anderson and Livingston \cite{anderson1999}, following Beck's earlier coloring construction \cite{beck1988}.  Recent work has investigated linear resolutions and Betti numbers of zero-divisor graphs \cite{rather2024_zd,dung2026,rather2026}.  We recover the complete-split structure of prime ideal graphs and derive the corresponding homological formulas, including the exceptional zero-prime case.  For finite products of Artinian chain rings, valuation vectors identify the sequentially Cohen--Macaulay co-chordal nilpotent zero-product graphs precisely with the threshold range.  Finally, among co-chordal zero-divisor graphs of $\mathbb Z_n$, we determine exactly when the quotient is sequentially Cohen--Macaulay and compute the Betti numbers in every surviving family.

The paper is organized as follows.  Section~\ref{sec:preliminaries} reviews glued clique complexes, stable connectivity, and $(d_1,\ldots,d_q)$-trees.  Section~\ref{sec:main} derives the Betti formula and its depth consequences.  Section~\ref{sec:rigidity} treats Cohen--Macaulay rigidity.  Section~\ref{sec:applications} considers split, threshold, and prime ideal graphs.  Section~\ref{sec:zero_divisors} studies nilpotent zero-product graphs of products of chain rings and zero-divisor graphs of $\mathbb Z_n$.  Section~\ref{sec:conclusion} concludes the paper.

We use the convention $\binom{a}{b}=0$ whenever $b<0$, $a<0$, or $b>a$.

\section{Glued Complexes and \texorpdfstring{$(d_1,\dots,d_q)$}{(d1,...,dq)}-Trees} \label{sec:preliminaries}
	
	The results developed in this paper are based on the combinatorial structure of glued clique complexes and their connection with sequentially Cohen--Macaulay co-chordal graphs. Throughout, let $G$ be a finite simple graph on $N$ vertices, and let $\operatorname{maxdeg}(G)$ and $\operatorname{indeg}(G)$ denote its maximum and minimum vertex degrees, respectively.
	
	\subsection{Glued Complexes and Betti Numbers}

A simplicial complex is a \emph{glued complex} if it is obtained by successively attaching simplices along simplex faces \cite{namiq2026_explicit}.  Let $\Delta_r(n_1,\ldots,n_e)$ have facets $S_1,\ldots,S_e$, where $|S_m|=n_m$, and suppose that the simplicial-complex intersection
\[
\langle S_1,\ldots,S_m\rangle\cap\langle S_{m+1}\rangle
\]
is an $(r_m-1)$-simplex for $1\le m<e$.  Here $\langle S_1,\ldots,S_m\rangle$ denotes the subcomplex generated by the listed facets.  The ordered sequences $(n_1,\ldots,n_e)$ and $(r_1,\ldots,r_{e-1})$ record the facet and intersection sizes.
	
	When $\Delta_r$ is the clique complex of the chordal complement $\overline{G}$ of a co-chordal graph $G$, the edge ideal $I(G)$ admits a $2$-linear resolution. The linear-strand Betti numbers satisfy $\beta_{i,i+1}(S/I(G))=\beta_i(S/I(G))$ and are given by \cite{namiq2026_explicit}
	\begin{equation}\label{eq:general_betti}
		\beta_i(S/I(G)) = \sum_{m=1}^{e-1}\binom{N-r_m}{i+1} - \sum_{m=1}^{e}\binom{N-n_m}{i+1}, \quad \text{for } i \ge1,
	\end{equation}
	where $N=|V(G)|$. The corresponding projective-dimension calculation in \cite{namiq2026_explicit} gives
	\[
\operatorname{pdim}(S/I(G))=N-\min_{1\le m\le e-1}r_m-1
\]
when $e\ge2$; if $e=1$, then $G$ is edgeless and the projective dimension is zero.
	
	\subsection{Stable Connectivity and \texorpdfstring{$(d_1,\ldots,d_q)$}{(d1,...,dq)}-Trees}
	
	Transitioning from arbitrary glued complexes to sequentially Cohen--Macaulay structures requires strict topological connectivity conditions. Recall that a \emph{ridge} of a facet $F$ is a face of cardinality $|F|-1$.
	
	\begin{lemma}\label{lem:ridge-obstruction}
Let $G$ be co-chordal.  If $\operatorname{Ind}(G)$ has two facets, each of which shares no ridge with any other facet, then $S/I(G)$ is not sequentially Cohen--Macaulay.
\end{lemma}
\begin{proof}
If $S/I(G)$ were sequentially Cohen--Macaulay, then \cite[Theorem~3.2]{ahmed2025} would imply that $\overline G$ is a $(d_1,\ldots,d_e)$-tree.  In a defining construction order, every facet except the initial one is attached along one of its ridges.  Hence at most one facet can fail to share a ridge with another facet, contrary to the hypothesis.
\end{proof}
	
	\begin{definition}
		A pure simplicial complex $\Delta$ is \emph{strongly connected} if, for any two facets $F$ and $F'$, there exists a sequence of facets 
		\[
		F=F_0,F_1,\ldots,F_k=F'
		\]
		such that $F_j\cap F_{j+1}$ is a face of codimension 1 (ridge) for every $j$.
	\end{definition}
	
	\begin{definition}[\cite{namiq2026_icm}]
		A simplicial complex $\Delta$ is \emph{stably connected} if every pure skeleton $\Delta^{[i]}$, $-1\le i\le \dim\Delta$, is strongly connected.
	\end{definition}
	
	Stable connectivity is necessary for sequential Cohen--Macaulayness \cite{namiq2026_icm}. Its graph-theoretic counterpart in the present setting is the class of $(d_1,\ldots,d_q)$-trees.
	
	\begin{definition}[\cite{ahmed2025}]
		Let $d_1\ge d_2\ge\cdots\ge d_q>0$. A graph $H$ is a $(d_1,\ldots,d_q)$-tree if it is constructed inductively by $H_1=K_{d_1}$, and
		\[
		H_i
		=
		H_{i-1}\cup_{K_{d_i-1}} K_{d_i},
		\qquad
		2\le i\le q.
		\]
		The final graph $H_q$ is denoted by $H$.
	\end{definition}
	
	The following characterization will be used throughout the paper.
	
	\begin{theorem}\label{thm:scm_equivalence}
Let $G$ be co-chordal, and write $\Delta=\operatorname{Ind}(G)$.  The following conditions are equivalent.
\begin{enumerate}
\item $\Delta$ is vertex decomposable.
\item $\Delta$ is shellable.
\item $S/I(G)$ is sequentially Cohen--Macaulay.
\item $\Delta$ is stably connected.
\item $\overline G$ is a $(d_1,\ldots,d_e)$-tree.
\item The maximal cliques of $\overline G$ admit an order $F_1,\ldots,F_e$, with $|F_1|\ge\cdots\ge |F_e|$, such that
\[
F_m\cap(F_1\cup\cdots\cup F_{m-1})
\]
is a ridge of $F_m$ for every $m\ge2$.
\end{enumerate}
For such an order, if $n_m=|F_m|$ and
\[
r_{m-1}=\bigl|F_m\cap(F_1\cup\cdots\cup F_{m-1})\bigr|,
\]
then $r_{m-1}=n_m-1$ for $2\le m\le e$.
\end{theorem}
\begin{proof}
The equivalence of (1)--(5) follows from \cite[Theorem~3.2]{ahmed2025} together with the stable-connectivity characterization in \cite[Corollary~5.6]{namiq2026_icm}.  Conditions (5) and (6) are equivalent by the definition of a $(d_1,\ldots,d_e)$-tree: the clique $F_m\cong K_{d_m}$ is attached to the preceding union along $K_{d_m-1}$.  The final numerical assertion is precisely the cardinality formulation of the ridge condition.
\end{proof}
	
	\section{Algebraic Simplification of the Betti Formula}\label{sec:main}
	
	The characterization of sequentially Cohen--Macaulay co-chordal graphs by $(d_1,\ldots,d_q)$-trees allows the general Betti formula for glued complexes to be simplified considerably. The resulting expression depends only on the maximal clique sizes of the chordal complement and provides a homological characterization of the sequentially Cohen--Macaulay property.
	
	\begin{theorem}\label{thm:seq_cm_betti_iff}
Let $G$ be co-chordal on $N$ vertices.  Then $S/I(G)$ is sequentially Cohen--Macaulay if and only if the maximal cliques of $\overline G$ admit a gluing order $F_1,\ldots,F_e$, with $n_m=|F_m|$ and $n_1\ge\cdots\ge n_e$, for which
\begin{equation}\label{eq:scm-betti}
\beta_i(S/I(G))
=
\sum_{m=2}^{e}\binom{N-n_m}{i}
-
\binom{N-n_1}{i+1}
\qquad(i\ge1).
\end{equation}
For $e=1$, both sums are interpreted as zero.
\end{theorem}
\begin{proof}
Fix a gluing order and put
\[
r_{m-1}=\bigl|F_m\cap(F_1\cup\cdots\cup F_{m-1})\bigr|.
\]
The general glued-complex formula \eqref{eq:general_betti} gives
\begin{equation}\label{eq:general-ordered}
\beta_i(S/I(G))
=
\sum_{m=2}^{e}\binom{N-r_{m-1}}{i+1}
-
\sum_{m=1}^{e}\binom{N-n_m}{i+1}.
\end{equation}
If $S/I(G)$ is sequentially Cohen--Macaulay, choose the order supplied by Theorem~\ref{thm:scm_equivalence}.  Then $r_{m-1}=n_m-1$, and Pascal's identity transforms \eqref{eq:general-ordered} into \eqref{eq:scm-betti}.

Conversely, suppose a gluing order satisfies \eqref{eq:scm-betti}.  Comparing it with \eqref{eq:general-ordered} yields
\[
\sum_{m=2}^{e}\binom{N-r_{m-1}}{k}
=
\sum_{m=2}^{e}\binom{N-n_m+1}{k}
\qquad(k\ge2).
\]
A finite multiset of integers at least $2$ is determined by all sums $\sum_j\binom{a_j}{k}$ for $k\ge2$: compare the largest upper index and its multiplicity, remove those terms, and continue by downward induction.  Hence
\[
\{N-r_{m-1}:2\le m\le e\}
=
\{N-n_m+1:2\le m\le e\}
\]
as multisets.  Equality of their sums gives
\[
\sum_{m=2}^{e}r_{m-1}=\sum_{m=2}^{e}(n_m-1).
\]
Because each $F_m$ is maximal, $r_{m-1}\le n_m-1$.  All these inequalities must therefore be equalities.  Theorem~\ref{thm:scm_equivalence} now implies that $S/I(G)$ is sequentially Cohen--Macaulay.  The case $e=1$ is the edgeless graph and is immediate.
\end{proof}
	
	Theorem~\ref{thm:seq_cm_betti_iff} reduces the general alternating Betti formula to an expression depending only on the maximal clique sizes of the chordal complement. In particular,  the projective dimension is determined by the smallest maximal clique. The following result recovers \cite[Theorem 5.4]{ahmed2025} and is also contained in \cite[Corollary 5.11]{namiq2026_icm}.
	
	\begin{corollary}\label{cor:pdim_max_degree}
Let $G$ be a sequentially Cohen--Macaulay co-chordal graph on $N$ vertices, and let $n_{\min}$ be the smallest size of a maximal clique of $\overline G$.  Then
\[
\operatorname{pdim}(S/I(G))=N-n_{\min}=\operatorname{maxdeg}(G)
\]
and
\[
\operatorname{depth}(S/I(G))=n_{\min}=N-\operatorname{maxdeg}(G).
\]
\end{corollary}
\begin{proof}
If $\overline G$ has one maximal clique, then $G$ is edgeless, so the projective dimension and the maximum degree are zero, while $n_{\min}=N$ and $\operatorname{depth}(S/I(G))=N$.

Assume that $\overline G$ has at least two maximal cliques and choose the ridge-gluing order supplied by Theorem~\ref{thm:scm_equivalence}.  Since the clique sizes are nonincreasing, a clique of size $n_{\min}$ occurs among $F_2,\ldots,F_e$.  Formula~\eqref{eq:scm-betti} shows that the largest homological index at which a positive summand can occur is $N-n_{\min}$.  At this index, the negative term vanishes, because
\[
\binom{N-n_1}{N-n_{\min}+1}=0,
\]
whereas a summand corresponding to a smallest maximal clique equals $1$.  Hence
\[
\operatorname{pdim}(S/I(G))=N-n_{\min}.
\]

The new vertex introduced by the final clique $F_e$, whose size is $n_{\min}$, has degree $n_{\min}-1$ in $\overline G$, and every vertex of $\overline G$ has degree at least this value.  Therefore
\[
\operatorname{indeg}(\overline G)=n_{\min}-1,
\]
so complementation gives
\[
\operatorname{maxdeg}(G)
=N-1-\operatorname{indeg}(\overline G)
=N-n_{\min}.
\]
Finally, the Auslander--Buchsbaum formula yields
\[
\operatorname{depth}(S/I(G))
=N-\operatorname{pdim}(S/I(G))
=n_{\min}
=N-\operatorname{maxdeg}(G).
\]
\end{proof}

\section{Connectivity and Cohen--Macaulay Homological Rigidity} \label{sec:rigidity}
	
	In the Cohen--Macaulay case, purity and codimension-one connectivity rigidify the clique construction.  The following characterization is a specialization of Theorem~\ref{thm:scm_equivalence}; compare \cite[Corollary~3.4]{ahmed2025} and \cite[Corollary~5.5]{namiq2026_icm}.
	
	\begin{theorem}\label{thm:cm_characterization}
Let $G$ be co-chordal.  The following are equivalent.
\begin{enumerate}
\item $S/I(G)$ is Cohen--Macaulay.
\item $\operatorname{Ind}(G)$ is pure and strongly connected in codimension one.
\item $\overline G$ is a pure $d$-tree.
\item The maximal cliques admit a construction order in which $n_1=\cdots=n_e=d$ and $r_1=\cdots=r_{e-1}=d-1$.
\end{enumerate}
\end{theorem}
\begin{proof}
If $S/I(G)$ is Cohen--Macaulay, then $\Delta=\operatorname{Ind}(G)$ is pure and Cohen--Macaulay; in particular, it is strongly connected in codimension one.  It is also sequentially Cohen--Macaulay, so Theorem~\ref{thm:scm_equivalence} gives a $(d_1,\ldots,d_e)$-tree construction.  Purity forces $d_1=\cdots=d_e=d$, and hence $\overline G$ is a pure $d$-tree.

Conversely, assume that $\Delta$ is pure and strongly connected in codimension one.  Every pure skeleton of $\Delta$ is strongly connected: two faces of the same dimension may first be connected inside their containing facets, and consecutive facets in a ridge chain may be crossed through faces contained in their common ridge.  Thus $\Delta$ is stably connected.  Theorem~\ref{thm:scm_equivalence} implies that $S/I(G)$ is sequentially Cohen--Macaulay.  Since $\Delta$ is pure, sequential Cohen--Macaulayness is equivalent to Cohen--Macaulayness.  This proves the equivalence of (1) and (2), while Theorem~\ref{thm:scm_equivalence} and purity give the equivalence with (3) and (4).
\end{proof}
	
	\begin{theorem}\label{thm:cm_betti_iff}
Let $G$ be co-chordal, and suppose that $\overline G$ has $e$ maximal cliques.  Then $S/I(G)$ is Cohen--Macaulay if and only if
\[
\beta_i(S/I(G))=i\binom{e}{i+1}
\qquad(i\ge1).
\]
\end{theorem}
\begin{proof}
If $S/I(G)$ is Cohen--Macaulay, then $\overline G$ is a pure $d$-tree.  Each clique after the first introduces one vertex, so $N-d=e-1$.  Formula \eqref{eq:scm-betti} gives
\[
\beta_i=(e-1)\binom{e-1}{i}-\binom{e-1}{i+1}
=i\binom{e}{i+1}.
\]

Conversely, assume the displayed Betti formula.  The resolution is $2$-linear, and its Hilbert numerator is
\[
1+\sum_{i=1}^{e-1}(-1)^i i\binom{e}{i+1}t^{i+1}
=(1+(e-1)t)(1-t)^{e-1}.
\]
Consequently,
\[
\dim S/I(G)=N-e+1.
\]
The last nonzero Betti number is $\beta_{e-1}=e-1$ when $e\ge2$, so the Auslander--Buchsbaum formula gives
\[
\operatorname{depth}S/I(G)=N-(e-1)=N-e+1.
\]
Thus depth equals dimension, and $S/I(G)$ is Cohen--Macaulay.  If $e=1$, then $G$ is edgeless and the assertion is immediate.
\end{proof}
	
	A \emph{co-tree} is a graph whose complement is a tree. The next result shows that its homology depends only on its size.
	
	\begin{corollary}
		Let $G$ be a co-tree on $N\ge 3$ vertices. Then $S/I(G)$ is Cohen--Macaulay and
		\[
		\beta_i(S/I(G)) = i\binom{N-1}{i+1}, \quad \text{for } i \ge1.
		\]
	\end{corollary}
	\begin{proof}
		A tree on $N$ vertices has $N-1$ edges, each corresponding to a maximal clique of size $2$. Hence $\overline{G}$ forms a pure $(2,\dots,2)$-tree, so Theorem~\ref{thm:cm_betti_iff} applies with $e=N-1$, giving the stated formula.
	\end{proof}
	
	\section{Applications to Classical Graph Families} \label{sec:applications}
	
	The algebraic framework from Sections \ref{sec:main} and \ref{sec:rigidity} yields exact homological invariants for several standard graph classes by expressing their chordal complements as $(d_1,\dots,d_q)$-trees.
	
	\subsection{Split Graphs}
	
	A graph $G$ is a \emph{split graph} if its vertex set admits a partition
	\[
	V(G)=C \sqcup U,
	\]
	where $C$ induces a complete subgraph and $U$ is an independent set.
	
	The defining structure of a split graph immediately implies that both the graph and its complement are chordal. Thus, the classical characterization of F\"oldes and Hammer \cite{foldes1977} is recovered as a consequence of the following result.
	
	\begin{theorem}\label{thm:split_is_scm}
Every split graph is co-chordal and sequentially Cohen--Macaulay.
\end{theorem}
\begin{proof}
Let $V(G)=C\sqcup U$, where $C=\{x_1,\ldots,x_n\}$ is a clique and $U$ is independent.  In $\overline G$, the set $U$ is a clique and $C$ is independent.  For each $k$, put
\[
F_k=\{x_k\}\cup\bigl(U\setminus N_G(x_k)\bigr).
\]
These are the maximal cliques containing vertices of $C$.  Order them by nondecreasing values of $|N_G(x_k)\cap U|$, so that their cardinalities are nonincreasing.  If every $N_G(x_k)\cap U$ is nonempty, then $U$ is also maximal; start with $U$ and attach the ordered cliques $F_k$ along the ridges $U\setminus N_G(x_k)$.

If $N_G(x_j)\cap U=\varnothing$ for some $j$, then $U$ is not maximal.  Start with one clique $F_j=U\cup\{x_j\}$ having zero cross-degree, place any remaining zero-cross-degree cliques next, and then append the other $F_k$ in the order just specified.  Every appended $F_k$ meets the preceding union in $U\setminus N_G(x_k)$, which is a ridge of $F_k$.  In either case the maximal cliques admit the order in Theorem~\ref{thm:scm_equivalence}; hence $\overline G$ is chordal and $S/I(G)$ is sequentially Cohen--Macaulay.
\end{proof}
	
	\begin{theorem}\label{thm:split_betti}
Let $G$ be a split graph with $V(G)=C\sqcup U$, where $C=\{x_1,\ldots,x_n\}$ and $n\ge1$.  Put
\[
n_k=|N_G(x_k)\cap U|.
\]
Then
\[
\beta_i(S/I(G))
=
\sum_{k=1}^{n}\binom{n+n_k-1}{i}
-
\binom{n}{i+1}
\qquad(i\ge1).
\]
\end{theorem}
\begin{proof}
Write $m=|U|$ and $N=n+m$.  If all $n_k>0$, the maximal cliques of $\overline G$ are $U$ and the sets
\[
F_k=\{x_k\}\cup(U\setminus N_G(x_k)),
\qquad |F_k|=m-n_k+1.
\]
Begin with $U$ and order the remaining cliques by nondecreasing $n_k$.  Formula~\eqref{eq:scm-betti} then gives the asserted expression.

Suppose next that $n_j=0$ for some $j$.  Start with $F_j=U\cup\{x_j\}$, place the remaining cliques in nonincreasing cardinality, and apply formula~\eqref{eq:scm-betti} to obtain
\[
\beta_i
=
\sum_{k\ne j}\binom{n+n_k-1}{i}
-
\binom{n-1}{i+1}.
\]
Since
\[
-\binom{n-1}{i+1}
=
\binom{n-1}{i}-\binom{n}{i+1}
\]
and $n_j=0$, this is again the stated formula.
\end{proof}

\begin{remark}\label{rem:split-literature}
	The formula in Theorem~\ref{thm:split_betti} is obtained directly from the ridge-gluing description of the maximal cliques of the complement.  Betti numbers of split graphs have also been studied by Singh and Verma \cite{singh2020}, and Fr\"oberg also attempted to work on split graphs \cite{froeberg2026}.  The present formulation uses the $(d_1,\ldots,d_q)$-tree framework and includes vertices of zero cross-degree.
\end{remark}
	
	\begin{corollary}\label{cor:split_pdim}
Under the hypotheses of Theorem~\ref{thm:split_betti},
\[
\operatorname{pdim}(S/I(G))
=n+\max_{1\le k\le n}n_k-1.
\]
\end{corollary}
\begin{proof}
Every $x_k\in C$ has degree $(n-1)+n_k$, whereas every vertex of $U$ has degree at most $n$.  If all $n_k=0$, the maximum degree is $n-1$; otherwise a clique vertex of maximum cross-degree has degree at least $n$.  Thus
\[
\operatorname{maxdeg}(G)=n+\max_k n_k-1,
\]
and Corollary~\ref{cor:pdim_max_degree} applies.
\end{proof}
	
	\begin{corollary}\label{cor:cm_split}
		Let $G$ be a split graph with vertex partition $V(G)=C \sqcup U$, and let $|C|=n$ and $n_k= |N_G(x_k)\cap U|$ for each $x_k \in C$.
		Then $S/I(G)$ is Cohen--Macaulay if and only if either $n_k=0$ for all $k$, or $n_k=1$ for all $k$. In these cases, the nonzero graded Betti numbers satisfy $\beta_{i,i+1}=\beta_i$ and are given by
		\[
		\beta_i(S/I(G)) = 
		\begin{cases}
			i\binom{n}{i+1} & \text{if } n_k=0 \text{ for all } k, \\[6pt]
			i\binom{n+1}{i+1} & \text{if } n_k=1 \text{ for all } k,
		\end{cases}
		\]
		for all $i\ge1$.
	\end{corollary}
	
	\begin{proof}
		By Theorem~\ref{thm:cm_characterization}, the quotient $S/I(G)$ is Cohen--Macaulay if and only if $\overline{G}$ is a pure $d$-tree. Let $V(G)=C\sqcup U$, where $C$ is a clique and $U$ is an independent set in $G$. Then $U$ induces a clique in $\overline{G}$. For each vertex $x_k\in C$, define
		\[
		F_k=\{x_k\}\cup\bigl(U\setminus N_G(x_k)\bigr).
		\]
		Since $x_k$ is adjacent in $\overline{G}$ precisely to the vertices of $U\setminus N_G(x_k)$, each $F_k$ is a maximal clique of $\overline{G}$ with
		\[
		|F_k|=|U|-n_k+1,
		\]
		where $n_k=|N_G(x_k)\cap U|$.
		
		Assume first that $n_j=0$ for some $j$. Then $U\subset F_j$, so $U$ is not a maximal clique. The maximal cliques are $F_1,\ldots,F_n$. Since $|F_j|=|U|+1$ and $\overline G$ is pure, every $F_k$ must have cardinality $|U|+1$. Hence $n_k=0$ for all $k$, and there are exactly $n$ maximal cliques.
		
		Assume next that $n_k>0$ for every $k$. Then no $F_k$ contains $U$, so $U$ is a maximal clique of cardinality $|U|$. Purity gives
		\[
		|U|-n_k+1=|U|
		\]
		for every $k$, and therefore $n_k=1$ for all $k$. In this case the maximal cliques are $U,F_1,\ldots,F_n$, so there are exactly $n+1$ of them.
		
		Therefore, $\overline{G}$ is pure if and only if either $n_k=1$ for all $k$ or $n_k=0$ for all $k$. The corresponding numbers of facets are $n+1$ and $n$, respectively. Substituting these values into the Betti number formula of Theorem~\ref{thm:cm_betti_iff} yields
		\[
		\beta_i(S/I(G))=i\binom{n+1}{i+1}
		\]
		when $n_k=1$ for all $k$, and
		\[
		\beta_i(S/I(G))=i\binom{n}{i+1}
		\]
		when $n_k=0$ for all $k$, completing the proof.
	\end{proof}
	
	\begin{remark}
		Villarreal \cite{villarreal1990} proved that attaching a whisker (a pendant edge) to every vertex of a graph yields a Cohen--Macaulay ring. When the base graph is a complete graph $K_n$, the resulting whisker graph is a split graph. Our characterization establishes the converse: a split graph where $|C| = |U|$ (with no isolated vertices) is Cohen--Macaulay if and only if it arises as the whisker graph of a complete graph. Hence, within this symmetric subclass of split graphs, Villarreal's construction precisely characterizes Cohen--Macaulayness.
	\end{remark}
	
	\begin{example}
		Consider the split graph $G$ on $N=7$ vertices shown in Figure~\ref{fig:split_example}. Its vertex set is partitioned into a clique $C=\{x_1,x_2,x_3\}$ of size $n=3$ and an independent set $U=\{y_1,y_2,y_3,y_4\}$ of size $m=4$. From the graph, the bipartite neighborhood sizes are:
		\[
		n_1=|N_G(x_1)\cap U|=4,\quad n_2=|N_G(x_2)\cap U|=2,\quad n_3=|N_G(x_3)\cap U|=1.
		\]
		
		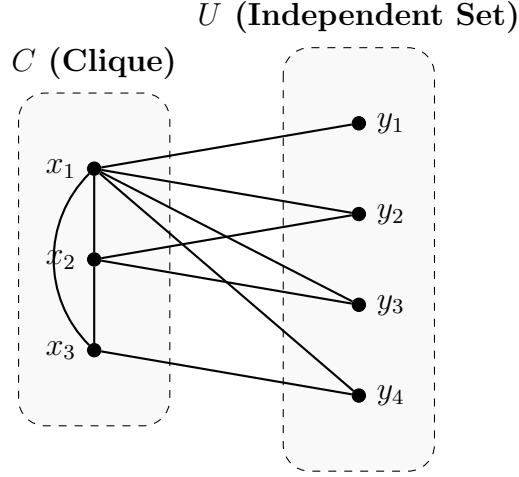
\begin{figure}[htpb]
			\centering
			\begin{tikzpicture}[
				vertex/.style={circle, draw, fill=black, inner sep=0pt, minimum size=5pt},
				edge/.style={thick}
				]
				
				\draw[dashed, rounded corners=10pt, fill=gray!5] (-1.0, 2.2) rectangle (1.0, -2.2);
				\node[font=\normalsize\bfseries] at (0, 2.6) {$C$ (Clique)};
				
				\draw[dashed, rounded corners=10pt, fill=gray!5] (2.5, 2.8) rectangle (4.5, -2.8);
				\node[font=\normalsize\bfseries] at (3.5, 3.2) {$U$ (Independent Set)};
				
				\node[vertex, label=left:{$x_1$}] (x1) at (0, 1.2) {};
				\node[vertex, label=left:{$x_2$}] (x2) at (0, 0) {};
				\node[vertex, label=left:{$x_3$}] (x3) at (0, -1.2) {};
				
				\node[vertex, label=right:{$y_1$}] (y1) at (3.5, 1.8) {};
				\node[vertex, label=right:{$y_2$}] (y2) at (3.5, 0.6) {};
				\node[vertex, label=right:{$y_3$}] (y3) at (3.5, -0.6) {};
				\node[vertex, label=right:{$y_4$}] (y4) at (3.5, -1.8) {};
				
				\draw[edge] (x1) -- (x2);
				\draw[edge] (x2) -- (x3);
				\draw[edge] (x1) edge[bend right=45] (x3);
				\draw[edge] (x1) -- (y1);
				\draw[edge] (x1) -- (y2);
				\draw[edge] (x1) -- (y3);
				\draw[edge] (x1) -- (y4);
				\draw[edge] (x2) -- (y2);
				\draw[edge] (x2) -- (y3);
				\draw[edge] (x3) -- (y4);
				
			\end{tikzpicture}
			\vspace{0.3cm} 
			\caption{The split graph structure of $G$.}
			\label{fig:split_example}
		\end{figure}
		
		Applying Theorem~\ref{thm:split_betti} to $G$ yields the expression
		\[
		\beta_i(S/I(G)) = \binom{6}{i}+\binom{4}{i}+\binom{3}{i}-\binom{3}{i+1}, \quad \text{for } i \ge1.
		\]
		Since $\beta_0(S/I(G))=1$, evaluating this expression yields the following Betti table for $S/I(G)$:
		\begin{center}
			\begin{tabular}{r|ccccccc} 
				& 0 & 1 & 2 & 3 & 4 & 5 & 6 \\
				\text{total:} & 1 & 10 & 23 & 25 & 16 & 6 & 1 \\\hline
				0 & 1 & $\cdot$ & $\cdot$ & $\cdot$ & $\cdot$ & $\cdot$ & $\cdot$ \\
				1 & $\cdot$ & 10 & 23 & 25 & 16 & 6 & 1
			\end{tabular}
		\end{center}
		
		Since $\beta_6\neq 0$ and $\beta_i=0$ for all $i>6$, it follows that $\operatorname{pdim}(S/I(G))=6$. This agrees with Corollary~\ref{cor:split_pdim}, which establishes $\operatorname{pdim}(S/I(G)) = n + \max\{n_k\} - 1 = 3 + 4 - 1 = 6$.
	\end{example}
	
	\subsection{Complete, Single-Edge-Deleted, and Pineapple Graphs}
	
	The exact expression in Theorem~\ref{thm:split_betti} admits several useful specializations corresponding to uniform or almost uniform bipartite attachments between the clique and the independent set.  The following formulas are direct consequences of that theorem; see Remark~\ref{rem:split-literature} for related literature.
	
	\begin{corollary}[Complete split graphs]\label{cor:complete_split}
		Let $G$ be a complete split graph, i.e., every vertex of $C$ is adjacent to every vertex of $U$. If $|C|=n$ and $|U|=m$, then
		\[
		\beta_i(S/I(G)) = n\binom{n+m-1}{i} - \binom{n}{i+1}, \quad \text{for } i \ge1.
		\]
	\end{corollary}
	
	\begin{proof}
		In a complete split graph, each vertex $x_k \in C$ satisfies $n_k=m$ for all $1 \le k \le n$. Hence,
		\[
		n+n_k-1=n+m-1,\qquad 1\le k\le n.
		\]
		Therefore the summation in Theorem~\ref{thm:split_betti} becomes
		\[
		\sum_{k=1}^n \binom{n+n_k-1}{i} = \sum_{k=1}^n \binom{n+m-1}{i}
		= n\binom{n+m-1}{i},
		\]
		yielding the stated formula.
	\end{proof}
	
	\begin{corollary}[Single-edge-deleted complete split graphs]
		Let $G$ be a split graph in which exactly one edge between $C$ and $U$ is missing. Then
		\[
		\beta_i(S/I(G)) = (n-1)\binom{n+m-1}{i} + \binom{n+m-2}{i} - \binom{n}{i+1}, \quad \text{for } i \ge1.
		\]
	\end{corollary}
	
	\begin{proof}
		Exactly one vertex, say $x_1 \in C$, satisfies $n_1=m-1$, while $n_k=m$ for all $k\ge 2$. Hence the summation splits as
		\[
		\sum_{k=1}^n \binom{n+n_k-1}{i}
		= \binom{n+m-2}{i} + \sum_{k=2}^n \binom{n+m-1}{i}.
		\]
		The second term contributes $(n-1)\binom{n+m-1}{i}$, and substituting into Theorem~\ref{thm:split_betti} gives the result.
	\end{proof}
	
	\begin{corollary}[Pineapple graphs]
		Let $\mathcal{P}_a^b$ be a pineapple graph consisting of a clique $K_a$ and a stable set of size $b$, where all vertices of the stable set are attached to a single fixed vertex in the clique. Then
		\[
		\beta_i(S/I(\mathcal{P}_a^b)) = \binom{a+b-1}{i} + (a-1)\binom{a-1}{i} - \binom{a}{i+1}, \quad \text{for } i \ge1.
		\]
	\end{corollary}
	
	\begin{proof}
		Let $C$ denote the clique with $|C|=a$. Exactly one vertex $x_1 \in C$ satisfies $n_1=b$, while all remaining vertices satisfy $n_k=0$ for $k\ge 2$. Hence
		\[
		\sum_{k=1}^a \binom{a+n_k-1}{i}
		= \binom{a+b-1}{i} + \sum_{k=2}^a \binom{a-1}{i}
		= \binom{a+b-1}{i} + (a-1)\binom{a-1}{i}.
		\]
		Substituting into Theorem~\ref{thm:split_betti} yields the stated expression.
	\end{proof}
	
	\begin{remark}
		Corollary~\ref{cor:cm_split} determines the Cohen--Macaulay members of the preceding subclasses. A complete split graph, for which $n_k=m$ for all $k$, is Cohen--Macaulay if and only if $m\in\{0,1\}$, and hence reduces to a pure complete graph in both cases.
		
		Single-edge-deleted complete split graphs and pineapple graphs fail the required uniform attachment condition due to non-constant attachment degrees. Consequently, a single-edge-deleted complete split graph with $n_1=m-1$ and $n_k=m$ for $k\ge2$ is Cohen--Macaulay only in the exceptional cases $n=1$ and $m\in\{1,2\}$. Likewise, a pineapple graph $\mathcal P_a^b$ with $n_1=b$ and $n_k=0$ for $k\ge2$ is Cohen--Macaulay precisely when $b=0$ or $(a,b)=(1,1)$.
	\end{remark}
	
	\subsection{Threshold Graphs}
	
	Threshold graphs form a distinguished subclass of split graphs. They may be constructed recursively from a single vertex by repeatedly adding either an isolated vertex or a dominating vertex \cite{nagel2009}. Equivalently, a split graph is threshold if and only if the neighborhoods of the clique vertices within the independent set are linearly ordered by inclusion.
	
	\begin{corollary}
		Let $G$ be a threshold graph with split partition $V(G)=C\sqcup U$. Let $C=\{x_1,\ldots,x_n\}$ be ordered so that $n_1\ge n_2\ge \cdots \ge n_n$, where $n_k=|N_G(x_k)\cap U|$. Then
		\[
		\beta_i(S/I(G)) = \sum_{k=1}^{n} \binom{n+n_k-1}{i} - \binom{n}{i+1}, \quad \text{for } i \ge1,
		\]
		and
		\[
		\operatorname{pdim}(S/I(G)) = n+n_1-1.
		\]
	\end{corollary}
	\begin{proof}
		Since every threshold graph is a split graph, the Betti formula follows directly from Theorem~\ref{thm:split_betti}. Moreover, Corollary~\ref{cor:split_pdim} yields
		\[
		\operatorname{pdim}(S/I(G)) = n+\max_{1\le k\le n}\{n_k\}-1.
		\]
		Because the neighborhoods of the vertices in $C$ are nested, the chosen ordering implies
		\[
		n_1=\max_{1\le k\le n}\{n_k\},
		\]
		and the stated formula follows.
	\end{proof}
	
	\subsection{Prime Ideal Graphs}
	
	We next apply the preceding results to prime ideal graphs of finite commutative rings. Let $R$ be a finite commutative ring and let $P$ be a proper prime ideal of $R$. The \emph{prime ideal graph} $\Gamma_P(R)$ is the graph whose vertex set is $R\setminus\{0\}$, where two distinct vertices $x$ and $y$ are adjacent whenever $xy\in P$ \cite{salih2022,rasheed2026}.
	
	\begin{proposition}\label{prop:pig_complete_split}
Let $R$ be a finite commutative ring and $P$ a proper prime ideal.  Then
\[
\Gamma_P(R)\cong K_{|P|-1}\vee\overline K_{|R|-|P|}.
\]
In particular, $\Gamma_P(R)$ is a complete split graph.
\end{proposition}
\begin{proof}
Put $C=P\setminus\{0\}$ and $U=R\setminus P$.  Products of two elements of $C$, and products of an element of $C$ with an element of $U$, lie in $P$.  Thus $C$ is a clique and every edge between $C$ and $U$ is present.  If $x,y\in U$, primality of $P$ implies $xy\notin P$, so $U$ is independent.  The stated join decomposition follows.  This description also appears in \cite{rasheed2026}.
\end{proof}
	
	The homological invariants now follow immediately from Corollaries \ref{cor:split_pdim} and \ref{cor:complete_split}.
	
	\begin{theorem}
Let $R$ be a finite commutative ring and $P$ a proper prime ideal.  Then $S/I(\Gamma_P(R))$ is sequentially Cohen--Macaulay and
\[
\beta_i(S/I(\Gamma_P(R)))
=(|P|-1)\binom{|R|-2}{i}-\binom{|P|-1}{i+1}
\qquad(i\ge1).
\]
Moreover,
\[
\operatorname{pdim}(S/I(\Gamma_P(R)))=
\begin{cases}
0,&P=(0),\\
|R|-2,&P\ne(0).
\end{cases}
\]
\end{theorem}
\begin{proof}
If $P=(0)$, then the finite ring $R$ is a field.  The graph is edgeless, the displayed Betti formula is zero, and the projective dimension is zero.

Assume $P\ne(0)$.  Proposition~\ref{prop:pig_complete_split} gives a complete split graph with clique size $|P|-1$ and independent-set size $|R|-|P|$.  Corollary~\ref{cor:complete_split} yields the Betti formula, and Corollary~\ref{cor:split_pdim} gives $\operatorname{pdim}=|R|-2$.
\end{proof}
	
	\begin{corollary}\label{cor:cm_prime_ideal}
		Let $R$ be a finite commutative ring and $P$ a prime ideal. Then $\Gamma_P(R)$ is Cohen--Macaulay if and only if $R$ is a finite field.
	\end{corollary}
	
	\begin{proof}
		Suppose first that $\Gamma_P(R)$ is Cohen--Macaulay. If $P\neq (0)$, then Proposition~\ref{prop:pig_complete_split} yields
		\[
		\Gamma_P(R)\cong K_{|P|-1}\vee \overline{K}_{|R|-|P|}.
		\]
		By Corollary~\ref{cor:cm_split}, a Cohen--Macaulay split graph satisfies either $n_k=0$ for all $k$ or $n_k=1$ for all $k$. Since $\Gamma_P(R)$ is a complete split graph, every vertex of the clique is adjacent to every vertex of the independent set, and therefore
		\[
		n_k = |U| = |R| - |P| \quad \text{for all } k.
		\]
		Since $P$ is a proper ideal, $|R| > |P|$. Thus, the condition $|R|-|P| = 0$ is impossible, meaning the case $n_k=0$ cannot occur. Consequently, $n_k=1$ for all $k$, which implies
		\[
		|R|-|P|=1.
		\]
		Using $|R|=|P||R/P|$, we obtain
		\[
		|P|(|R/P|-1)=1,
		\]
		which has no solution with $|P|\ge 2$. Hence $P=(0)$. Therefore $(0)$ is a prime ideal, so $R$ is an integral domain. Since $R$ is finite, it follows that $R$ is a finite field.
		
		Conversely, let $R=\mathbb F_q$. Then $P=(0)$ is the unique prime ideal. Because every nonzero element of $\mathbb F_q$ is a unit, there are no nonzero zero-divisors. Consequently,
		\[
		\Gamma_{(0)}(\mathbb F_q)\cong \overline{K}_{q-1}.
		\]
		Its edge ideal is zero, and hence
		\[
		S/I(\Gamma_{(0)}(\mathbb F_q))
		\cong S,
		\]
		which is Cohen--Macaulay. Therefore $\Gamma_P(R)$ is Cohen--Macaulay.
	\end{proof}
	
	\begin{remark}
		While Corollary~\ref{cor:cm_prime_ideal} shows that $\Gamma_P(R)$ is Cohen--Macaulay whenever $R$ is a finite field $\mathbb F_q$, the resulting graph is the edgeless graph $\overline{K}_{q-1}$. In particular, it is disconnected for all $q>2$. If one further imposes connectivity, then necessarily $q-1=1$, i.e., $q=2$. Consequently, a prime ideal graph $\Gamma_P(R)$ is both Cohen--Macaulay and connected if and only if $R \cong \mathbb Z_2$.
	\end{remark}
	
	\section{Homological Invariants of Nilpotent and Zero-Divisor Graphs} \label{sec:zero_divisors}
	
	Let $R$ be a commutative ring with identity.  The zero-divisor graph $\Gamma(R)$, in the sense of Anderson and Livingston \cite{anderson1999} following Beck's earlier coloring formulation \cite{beck1988}, has the nonzero zero-divisors as vertices, with distinct vertices adjacent when their product is zero.  We write $\Gamma_{\operatorname{Nil}}(R)$ for the induced subgraph of $\Gamma(R)$ on $\operatorname{Nil}(R)\setminus\{0\}$; thus its vertices are the nonzero nilpotent elements and adjacency is again defined by $xy=0$.  This terminology is fixed for the present paper.  Recent work on linear resolutions and Betti numbers of zero-divisor graphs includes \cite{dung2026,rather2024_zd,rather2026}.  The structural results above permit exact calculations whenever a split or threshold decomposition is available.
	
	\subsection{Nilpotent Graphs of Artinian Rings}
	We characterize exactly when the nilpotent graph belongs to the sequentially Cohen--Macaulay co-chordal class and then compute its homological invariants in the equivalent threshold range.
	
\begin{theorem}\label{thm:master_nilpotent}
Let $R\cong\prod_{i=1}^{k}R_i$ be a finite direct product of commutative Artinian chain rings.  Let $\mathfrak m_i$ be the maximal ideal of $R_i$, let $a_i$ be its nilpotency index, and put $q_i=|R_i/\mathfrak m_i|$.  Set
\[
r=\bigl|\{i:a_i\ge3\}\bigr|.
\]
Then the following conditions are equivalent.
\begin{enumerate}
\item $\Gamma_{\operatorname{Nil}}(R)$ is co-chordal and $S/I(\Gamma_{\operatorname{Nil}}(R))$ is sequentially Cohen--Macaulay.
\item $\Gamma_{\operatorname{Nil}}(R)$ is a threshold graph.
\item At most one factor $R_i$ has $a_i\ge3$, equivalently $r\le1$.
\end{enumerate}

Suppose that these conditions hold and that exactly one factor, say $R_1$, has $a_1\ge3$.  Put
\[
M=\prod_{i=2}^{k}q_i^{a_i-1}.
\]
Then
\begin{equation}\label{eq:nil_product}
\beta_i(S/I(\Gamma_{\operatorname{Nil}}(R)))
=
\sum_{t=\lceil a_1/2\rceil}^{a_1-1}|W_t|
\binom{|C|+n_t-1}{i}
-
\binom{|C|}{i+1},
\end{equation}
where
\[
|W_t|=
\begin{cases}
M(q_1^{a_1-t}-q_1^{a_1-t-1}),&t<a_1-1,\\
Mq_1-1,&t=a_1-1,
\end{cases}
\]
and
\[
|C|=Mq_1^{\lfloor a_1/2\rfloor}-1,
\qquad
n_t=M(q_1^t-q_1^{\lfloor a_1/2\rfloor}).
\]
\end{theorem}
\begin{proof}
For each $i$, choose a generator $\pi_i$ of $\mathfrak m_i$ and define, for $x_i\in\mathfrak m_i$,
\[
\nu_i(x_i)=
\begin{cases}
t,&x_i\in\mathfrak m_i^t\setminus\mathfrak m_i^{t+1},\\
a_i,&x_i=0.
\end{cases}
\]
Since $R_i$ is a chain ring,
\begin{equation}\label{eq:valuation-zero-product}
x_iy_i=0
\quad\Longleftrightarrow\quad
\nu_i(x_i)+\nu_i(y_i)\ge a_i.
\end{equation}
The vertex set of $G:=\Gamma_{\operatorname{Nil}}(R)$ is
\[
V(G)=\left(\mathfrak m_1\times\cdots\times\mathfrak m_k\right)\setminus\{0\}.
\]
Consequently, two distinct vertices $x$ and $y$ are adjacent in $H:=\overline G$ if and only if
\begin{equation}\label{eq:nil-complement-adjacency}
\nu_i(x_i)+\nu_i(y_i)<a_i
\quad\text{for at least one }i.
\end{equation}
A coordinate with $a_i\le2$ never contributes an edge to $H$.

Assume first that condition~(3) holds.  If $r=0$, the product of any two nilpotent elements is zero, so $G$ is complete and therefore threshold.  If $r=1$, relabel the factors so that $a_1\ge3$.  Equation~\eqref{eq:valuation-zero-product} gives
\[
xy=0
\quad\Longleftrightarrow\quad
\nu_1(x_1)+\nu_1(y_1)\ge a_1,
\]
because the remaining coordinates always have zero product.  Thus $G$ has a threshold-weight representation with vertex weight $\nu_1(x_1)$ and threshold $a_1$.  This proves $(3)\Rightarrow(2)$.

Threshold graphs are closed under complementation, and every threshold graph is split.  Hence $G$ is co-chordal, and Theorem~\ref{thm:split_is_scm} implies that $S/I(G)$ is sequentially Cohen--Macaulay.  Therefore $(2)\Rightarrow(1)$.

We prove $(1)\Rightarrow(3)$.  Assume that $H$ is chordal and that $S/I(G)$ is sequentially Cohen--Macaulay.  Suppose, toward a contradiction, that $r\ge2$, and set
\[
s=\bigl|\{i:a_i\ge4\}\bigr|.
\]
We first use chordality to restrict the possible active nilpotency indices.

If $s\ge2$, choose active coordinates with indices $a,b\ge4$, and put
\[
u=\left\lceil\frac a2\right\rceil,
\qquad
v=\left\lceil\frac b2\right\rceil.
\]
Choose four vertices whose valuations in these two coordinates are
\[
(1,b-1),\qquad (u,v),\qquad (a-1,1),\qquad (u,v),
\]
using distinct elements for the two vertices with valuation pair $(u,v)$ and setting all other coordinates equal to zero.  By~\eqref{eq:nil-complement-adjacency}, consecutive vertices are adjacent in $H$, while the opposite pairs are nonadjacent.  They therefore induce a $4$-cycle, contrary to chordality.

If $s=1$ and $r\ge3$, choose active coordinates with indices $a\ge4,3,3$ and vertices with valuation triples
\[
(1,2,2),\qquad
\left(\left\lceil\frac a2\right\rceil,1,2\right),\qquad
(a-1,1,1),\qquad
\left(\left\lceil\frac a2\right\rceil,2,1\right).
\]
Again, consecutive pairs and only consecutive pairs satisfy~\eqref{eq:nil-complement-adjacency}; hence these vertices induce a $4$-cycle.

If $s=0$ and $r\ge4$, choose four active coordinates, all of nilpotency index $3$.  For a vertex $x$, let its shallow support be the set of these coordinates in which $\nu_i(x_i)=1$, taking valuation $2$ in the remaining active coordinates.  Vertices with shallow supports
\[
\{1,2\},\qquad \{2,3\},\qquad \{3,4\},\qquad \{4,1\}
\]
induce a $4$-cycle, because two such vertices are adjacent in $H$ exactly when their shallow supports intersect.

It follows that, under the assumption $r\ge2$, chordality leaves only the following possibilities:
\begin{enumerate}
\item $r\in\{2,3\}$ and every active index equals $3$;
\item $r=2$, one active index is at least $4$, and the other equals $3$.
\end{enumerate}
We show that both possibilities contradict sequential Cohen--Macaulayness.

In the first case, for every active coordinate $i$, define
\[
L_i=\{x\in V(H):\nu_i(x_i)=1\}.
\]
The set $L_i$ is a clique in $H$.  It is maximal: if $z\notin L_i$, choose a vertex $e$ whose $i$th coordinate has valuation $1$ and whose other coordinates are zero.  Then $e\in L_i$ and $e$ is nonadjacent to $z$.  Let $E_i$ be the set of all such vertices $e$.  Since
\[
|\mathfrak m_i\setminus\mathfrak m_i^2|=q_i^2-q_i\ge2,
\]
the set $E_i$ contains at least two vertices.  Every neighbor of a vertex in $E_i$ belongs to $L_i$; therefore every maximal clique different from $L_i$ omits all of $E_i$.  Thus $L_i$ shares no ridge with any other maximal clique.  Since there are at least two active coordinates, the clique complex of $H$, equivalently $\operatorname{Ind}(G)$, has at least two ridge-isolated facets.  This contradicts Lemma~\ref{lem:ridge-obstruction}.

In the second case, relabel so that $a_1=a\ge4$ and $a_2=3$, and put
\[
L=\{x:\nu_2(x_2)=1\},
\qquad
f=\left\lfloor\frac{a-1}{2}\right\rfloor,
\qquad
h=\left\lfloor\frac a2\right\rfloor.
\]
The set $L$ is a maximal clique.  Indeed, if $z\notin L$, then every vertex whose first coordinate is zero and whose second coordinate has valuation $1$ is nonadjacent to $z$.  This block contains at least two vertices, and every maximal clique distinct from $L$ omits it.  Hence $L$ is ridge-isolated.

Define
\[
Q=
\{x\notin L:\nu_1(x_1)\le f\}
\ \cup\
\{x\in L:\nu_1(x_1)\le h\}.
\]
The set $Q$ is a clique.  Two vertices in the first part are adjacent through the first coordinate because $2f<a$; two vertices in the second part are adjacent through the second coordinate; and a vertex from each part is adjacent through the first coordinate because $f+h<a$.

The clique $Q$ is maximal.  If $z\in L\setminus Q$, then $z$ is nonadjacent to every vertex outside $L$ whose first-coordinate valuation is $f$.  If $z\notin L\cup Q$, then $z$ is nonadjacent to every vertex in $L$ whose first-coordinate valuation is $h$.  Both boundary blocks contain at least two vertices.  Moreover, every maximal clique different from $Q$ contains a vertex outside $Q$ and consequently omits one of these two entire blocks.  Hence $Q$ is ridge-isolated.  The two ridge-isolated facets $L$ and $Q$ contradict Lemma~\ref{lem:ridge-obstruction}.

Both remaining cases are impossible; therefore $r\le1$.  This proves $(1)\Rightarrow(3)$ and completes the equivalence.

It remains to prove the Betti formula.  Assume that $R_1$ is the unique factor with $a_1\ge3$.  Since $|\mathfrak m_i|=q_i^{a_i-1}$, the last $k-1$ coordinates contribute the multiplicity
\[
M=\prod_{i=2}^{k}q_i^{a_i-1}.
\]
The vertices with first-coordinate valuation at least $\lceil a_1/2\rceil$, together with the vertices whose first coordinate is zero, form a clique $C$; the remaining vertices form an independent set.  Their cross-neighborhoods are nested by valuation.

In $R_1$, the valuation-$t$ layer has cardinality
\[
q_1^{a_1-t}-q_1^{a_1-t-1}.
\]
For $t<a_1-1$, this gives
\[
|W_t|=M(q_1^{a_1-t}-q_1^{a_1-t-1}).
\]
At $t=a_1-1$, the $M-1$ vertices with first coordinate zero are universal and join the $M(q_1-1)$ ordinary layer vertices, so
\[
|W_{a_1-1}|=Mq_1-1.
\]
Summing the clique layers gives
\[
|C|=Mq_1^{\lfloor a_1/2\rfloor}-1.
\]
A clique vertex of valuation $t$ is adjacent to exactly
\[
n_t=M(q_1^t-q_1^{\lfloor a_1/2\rfloor})
\]
vertices in the independent set.  Applying Theorem~\ref{thm:split_betti} proves~\eqref{eq:nil_product}.
\end{proof}

\begin{corollary}\label{cor:nil_cm_exact}
With the notation of Theorem~\ref{thm:master_nilpotent}, $S/I(\Gamma_{\operatorname{Nil}}(R))$ is Cohen--Macaulay if and only if $a_i\le2$ for every $i$.  In this case
\[
\Gamma_{\operatorname{Nil}}(R)
\cong
K_{\prod_{i=1}^{k}q_i^{a_i-1}-1}.
\]
\end{corollary}
\begin{proof}
If all $a_i\le2$, the product of any two nilpotent elements is zero, so the graph is complete and its independence complex is a pure zero-dimensional complex.

Conversely, suppose $a_j\ge3$ for some $j$.  Choose a nonzero element in $\mathfrak m_i^{a_i-1}$ in every nonfield factor and zero in every field factor.  The resulting nonzero nilpotent element annihilates every vertex, so it is universal in the graph and gives a singleton facet of the independence complex.  In the $j$th factor, the two distinct valuation-one elements $\pi_j$ and $\pi_j+\pi_j^{a_j-1}$ have nonzero product $\pi_j^2$.  They therefore form an independent pair in the graph, contained in a facet of size at least two.  The independence complex is not pure and hence cannot be Cohen--Macaulay.
\end{proof}
	
	We now specialize the threshold formula to a single finite chain ring.  The resulting expression agrees with the chain-ring formula in \cite{rather2026}. Recall that a finite commutative ring $R$ is a \emph{chain ring} if its ideals form a totally ordered set under inclusion:
	\[
	\mathfrak m \supset \mathfrak m^2 \supset \cdots \supset \mathfrak m^{a-1} \supset \mathfrak m^a=(0),
	\]
	where $\mathfrak m$ is the unique maximal ideal and $a$ is the nilpotency index.
	
	\begin{theorem}\label{thm:chain_rings}
		Let $R$ be a finite chain ring with maximal ideal $\mathfrak{m}$, residue field cardinality $q = |R/\mathfrak{m}|$, and nilpotency index $a$. Then $\Gamma(R)$ is a threshold graph, and its nonzero graded Betti numbers satisfy $\beta_{i,i+1}=\beta_i$ with
		\[
		\beta_i(S/I(\Gamma(R))) = \sum_{t=\lceil a/2\rceil}^{a-1} q^{a-t-1}(q-1) \binom{q^t - 2}{i} - \binom{q^{\lfloor a/2\rfloor} - 1}{i+1}, \quad \text{for } i \ge1.
		\]
	\end{theorem}
	
	\begin{proof}
		Since $R$ is a finite chain ring, it is a local Artinian ring whose ideals form the chain
		\[
		R \supset \mathfrak m \supset \mathfrak m^2 \supset \cdots \supset \mathfrak m^{a-1} \supset \mathfrak m^a=(0).
		\]
		In a local Artinian ring, every zero divisor is nilpotent. Hence $Z(R)=\operatorname{Nil}(R)=\mathfrak m$, and therefore the zero divisor graph and nilpotent graph coincide:
		\[
		\Gamma(R)=\Gamma_{\operatorname{Nil}}(R).
		\]
		
		Viewing $R$ as a direct product consisting of a single local factor, we have $k=1$. Consequently, Theorem~\ref{thm:master_nilpotent} applies directly and implies that $\Gamma(R)$ is a threshold graph.
		
		We now evaluate the Betti numbers. If $a=1$, then $R$ is a field and $\Gamma(R)$ is the empty graph. Hence $I(\Gamma(R))=(0)$, and all Betti numbers $\beta_{i,i+1}=\beta_i$ vanish for $i \ge 1$. The stated formula is consistent with this case, since the summation range is empty and the binomial term reduces to $\binom{q^0-1}{i+1} = 0$.
		
		If $a=2$, then $\Gamma(R)=K_{q-1}$, and Theorem~\ref{thm:cm_betti_iff} gives $\beta_i=i\binom{q-1}{i+1}$; this agrees with the displayed formula by Pascal's identity.

		Assume $a\ge3$.  The Betti numbers are obtained from Theorem~\ref{thm:master_nilpotent} with
		\[
		M=\prod_{i=2}^{k}q_i^{a_i-1}=1.
		\]
		Substituting $M=1$ gives
		\[
		|C|=q^{\lfloor a/2\rfloor}-1,\quad n_t=q^t-q^{\lfloor a/2\rfloor}.
		\]
		It remains to simplify the coefficients $|W_t|$. For $1\le t<a-1$, Theorem~\ref{thm:master_nilpotent} gives
		\[
		|W_t| = q^{a-t}-q^{a-t-1} = q^{a-t-1}(q-1),
		\]
		and this expression also holds for $t=a-1$, since it gives $q-1$. Moreover,
		\[
		|C|+n_t-1 = \bigl(q^{\lfloor a/2\rfloor}-1\bigr) + \bigl(q^t-q^{\lfloor a/2\rfloor}\bigr) -1 = q^t-2.
		\]
		Substituting the simplified expressions for $|W_t|$, $|C|$, and $|C|+n_t-1$ into the Betti formula of Theorem~\ref{thm:master_nilpotent}, we obtain
		\[
		\beta_i\!\left(S/I(\Gamma(R))\right)
		=
		\sum_{t=\lceil a/2\rceil}^{a-1}
		q^{a-t-1}(q-1)\binom{q^t-2}{i}
		-
		\binom{q^{\lfloor a/2\rfloor}-1}{i+1},
		\]
		for all $i\ge 1$, as claimed.
	\end{proof}
	
	\begin{example}
		We apply the explicit formula of Theorem~\ref{thm:chain_rings} to two classical families of finite chain rings. By extracting the residue field cardinality $q$ and the nilpotency index $a$, we obtain exact homological invariants without any algorithmic computation.
		
		\begin{enumerate}
			\item \textbf{Truncated polynomial rings.}
			For $R=\mathbb Z_p[x]/\langle x^c \rangle$, where $p$ is prime and $c\ge 2$, the unique maximal ideal is $\mathfrak m=\langle x\rangle$, and the ideals form the chain $\langle x\rangle \supset \langle x^2\rangle \supset \cdots \supset \langle x^{c-1}\rangle \supset (0)$. Hence, $R$ is a finite chain ring with nilpotency index $a=c$. The residue field is $R/\mathfrak{m} \cong \mathbb Z_p$, which gives $q = p$. Therefore,
			\[
			\beta_i(S/I(\Gamma(R))) = \sum_{t=\lceil c/2 \rceil}^{c-1} p^{c-t-1}(p-1) \binom{p^t - 2}{i} - \binom{p^{\lfloor c/2 \rfloor} - 1}{i+1}, \quad \text{for } i \ge1.
			\]
			
			\item \textbf{Gaussian quotient rings.}
			For $R=\mathbb Z_{2^m}[i]$, one may identify
\[
R\cong \mathbb Z[i]/(2^m)=\mathbb Z[i]/((1+i)^{2m}),
\]
because $2$ is associated to $(1+i)^2$ in the Gaussian integers.  Since $\mathbb Z[i]$ is a principal ideal domain, this is a finite chain ring with maximal ideal generated by the image of $1+i$, nilpotency index $a=2m$, and residue field of cardinality $q=2$.
			Substituting $q=2$ and $a=2m$, and using $\lfloor 2m/2 \rfloor = \lceil 2m/2 \rceil = m$, we obtain
			\[
			\beta_i(S/I(\Gamma(R))) = \sum_{t=m}^{2m-1} 2^{2m-t-1} \binom{2^t - 2}{i} - \binom{2^m - 1}{i+1}, \quad \text{for } i \ge1.
			\]
		\end{enumerate}
	\end{example}
	
	\begin{corollary}\label{cor:chain_cm_exact}
		Let $R$ be a finite chain ring with maximal ideal $\mathfrak{m}$ and nilpotency index $a$. Then $S/I(\Gamma(R))$ is Cohen--Macaulay if and only if $a\le 2$. Equivalently, $\Gamma(R)$ is a complete graph on $|\mathfrak{m}| - 1 = q^{a-1} - 1$ vertices.
	\end{corollary}
	
	\begin{proof}
		Since $R$ is a finite chain ring, it is a local Artinian ring. Consequently, every zero-divisor is nilpotent, and hence
		\[
		\Gamma(R)=\Gamma_{\operatorname{Nil}}(R).
		\]
		Applying Corollary~\ref{cor:nil_cm_exact} with $k=1$ shows that $S/I(\Gamma(R))$ is Cohen--Macaulay precisely when $a\le 2$.
		
		If $a=1$, then $R$ is a field and $\Gamma(R)$ is the empty graph. If $a=2$, then $\mathfrak m^2=(0)$, so every pair of nonzero zero-divisors annihilates each other, and consequently $\Gamma(R)$ is a complete graph on $|\mathfrak{m}| - 1 = q^{a-1} - 1$ vertices.
		
		Conversely, if $a\ge 3$, Corollary~\ref{cor:nil_cm_exact} implies that $S/I(\Gamma(R))$ is not Cohen--Macaulay.
	\end{proof}
	
	\subsection{Zero-Divisor Graphs of \texorpdfstring{$\mathbb Z_n$}{Z-n}}
	While Theorem~\ref{thm:master_nilpotent} characterizes the threshold nilpotent graphs of Artinian direct products, the full zero-divisor graph $\Gamma(R)$ of a non-local ring contains additional vertices (non-nilpotent zero-divisors) that significantly alter its topology. We now study the full zero-divisor graph of $\mathbb Z_n$ and classify the sequentially Cohen--Macaulay cases within the co-chordal range.
	
	\begin{theorem}\label{thm:scm_zdg_classification}
		Let $n\ge2$ and assume that $\Gamma(\mathbb Z_n)$ is co-chordal.  Then $S/I(\Gamma(\mathbb Z_n))$ is sequentially Cohen--Macaulay if and only if
		\[
		n\in\{p^a,2p,2p^2\},
		\]
		where $p$ is prime and is odd in the last two families.  Moreover, $\Gamma(\mathbb Z_{2p^2})$ is split, whereas $\Gamma(\mathbb Z_n)$ is threshold for $n=p^a$ and $n=2p$.  The nonzero graded Betti numbers satisfy $\beta_{i,i+1}=\beta_i$ and are as follows:
		\begin{enumerate}
			\item If $n=p$, then $\Gamma(\mathbb Z_{p})\cong \overline K_{0}$, and
			\[
			\beta_{i}(S/I(\Gamma(\mathbb Z_{p}))) =0, \quad \text{for } i \ge1.
			\]
			
			\item If $n=p^2$, then $\Gamma(\mathbb Z_{p^2})\cong K_{p-1}$, and
			\[
			\beta_i(S/I(\Gamma(\mathbb Z_{p^2}))) = i\binom{p-1}{i+1}, \quad \text{for } i \ge1.
			\]
			
			\item If $n=2p$, where $p$ is an odd prime, then $\Gamma(\mathbb Z_{2p})\cong K_{1,p-1}$, and
			\[
			\beta_i(S/I(\Gamma(\mathbb Z_{2p}))) = \binom{p-1}{i}, \quad \text{for } i \ge1.
			\]
			
			\item If $n=p^a$ with $a\ge 3$, then
			\[
			\beta_i(S/I(\Gamma(\mathbb Z_{p^a}))) = \sum_{t=\lceil a/2 \rceil}^{a-1} p^{a-t-1}(p-1) \binom{p^t - 2}{i} - \binom{p^{\lfloor a/2 \rfloor} - 1}{i+1}, \quad \text{for } i \ge1.
			\]
			
			\item If $n=2p^2$, where $p$ is an odd prime, then
			\[
			\beta_i(S/I(\Gamma(\mathbb Z_{2p^2}))) = \binom{p^2-1}{i} + (p-1)\binom{2p-2}{i}
			- \binom{p}{i+1}, \quad \text{for } i \ge1.
			\]
		\end{enumerate}
	\end{theorem}
	
	\begin{proof}
		$(\Longleftarrow)$ We first show that if $n \in \{p^a, 2p, 2p^2\}$, then the graph is sequentially Cohen--Macaulay and we compute its Betti numbers.
		
		\textbf{Case $n=p^a$:}
		Here $\mathbb Z_{p^a}$ is a finite chain ring. Its zero-divisor graph coincides with its nilpotent graph, which is a threshold graph by Theorem~\ref{thm:master_nilpotent}. Substituting $q=p$ into Theorem~\ref{thm:chain_rings} yields the stated Betti formulas for all $a \ge 1$. Since threshold graphs are co-chordal, their edge ideals are sequentially Cohen--Macaulay.
		
		\textbf{Case $n=2p$, $p$ odd:}
		The nonzero zero divisors of $\mathbb Z_{2p}$ are $\{p,2,4,\dots,2(p-1)\}$. Since
		\[
		p \cdot (2k) \equiv 0 \pmod{2p}, \qquad (2i)(2j)\not\equiv 0,
		\]
		the vertex $p$ is adjacent to all even elements, while no two even elements annihilate each other. Thus,
		\[
		\Gamma(\mathbb Z_{2p}) \cong K_{1,p-1},
		\]
		a star graph. Star graphs are threshold graphs. Applying Theorem~\ref{thm:split_betti} gives the stated Betti formula.
		
		\textbf{Case $n=2p^2$, $p$ odd:}
		By the Chinese Remainder Theorem, $\mathbb Z_{2p^2}\cong\mathbb F_2\times\mathbb Z_{p^2}$. Partition the nonzero zero-divisors into:
		\begin{align*}
			X_0&=\{(0,u):u\in\mathbb Z_{p^2}^{\times}\}, & X_1&=\{(0,z):0\ne z\in p\mathbb Z_{p^2}\},\\
			Y_0&=\{(1,0)\}, & Y_1&=\{(1,z):0\ne z\in p\mathbb Z_{p^2}\}.
		\end{align*}
		Their cardinalities are $|X_0|=p^2-p$, $|X_1|=p-1$, $|Y_0|=1$, and $|Y_1|=p-1$.
		Set $C=X_1\cup Y_0$ and $U=X_0\cup Y_1$. The set $C$ is a clique and $U$ is an independent set. Thus $C\sqcup U$ is a split partition with $|C|=p$ and $|U|=p^2-1$. Since every split graph is co-chordal and sequentially Cohen--Macaulay, this establishes the property. 
		
		To compute the Betti numbers via Theorem~\ref{thm:split_betti}, we find the cross-degrees. The vertex $(1,0) \in C$ is adjacent to every vertex of $X_0$ and no vertex of $Y_1$, yielding cross-degree $n_{(1,0)} = p^2-p$. Each vertex of $X_1 \in C$ is adjacent to every vertex of $Y_1$ and no vertex of $X_0$, yielding cross-degree $n_{x} = p-1$. Substituting these sizes into Theorem~\ref{thm:split_betti} yields:
		\[
		\beta_i(S/I(\Gamma(\mathbb Z_{2p^2}))) = \binom{p+(p^2-p)-1}{i} + (p-1)\binom{p+(p-1)-1}{i} - \binom{p}{i+1}.
		\]
		This simplifies exactly to the stated Betti formula.
		
		$(\Longrightarrow)$ Conversely, assume that $\Gamma(\mathbb Z_n)$ is co-chordal and $S/I(\Gamma(\mathbb Z_n))$ is sequentially Cohen--Macaulay. By the co-chordal classification of Dung and Vu \cite{dung2026}; see also \cite{rather2026}, $n$ must have one of the forms $p^a$, $p^a q$, or $pqr$, where $p,q,r$ are distinct primes. Prime powers $p^a$ are classified above. We rule out the remaining cases using Lemma~\ref{lem:ridge-obstruction}.
		
		\textbf{The case $n=p^a q$:} 
		If $a=1$, then $\mathbb Z_{pq}\cong\mathbb F_p\times\mathbb F_q$ and $\Gamma(\mathbb Z_{pq})=K_{p-1,q-1}$. Its independence complex consists of two disjoint simplices, of cardinalities $p-1$ and $q-1$. If both primes are odd, these two facets have cardinality at least two and neither shares a ridge with the other. Lemma~\ref{lem:ridge-obstruction} rules out sequential Cohen--Macaulayness. Hence one of the primes must be $2$, which gives $n=2p$.
		
		Now suppose $a\ge2$. Via $\mathbb Z_{p^a q}\cong\mathbb Z_{p^a}\times\mathbb F_q$, write $\nu(x)=\nu_p(x)$ for $x\in\mathbb Z_{p^a}$, with $\nu(0)=a$. The vertex set is the disjoint union of $A_i=\{(x,0):\nu(x)=i\}$ ($0\le i\le a-1$) and $B_i=\{(x,y):\nu(x)=i,\ y\in\mathbb F_q^\times\}$ ($1\le i\le a$). In the complement $H=\overline{\Gamma(\mathbb Z_{p^a q})}$, consider the two maximal facets:
		\[
		\mathcal B=B_1\cup\cdots\cup B_a,
		\qquad
		F_0=A_0\cup B_1\cup\cdots\cup B_{a-1}.
		\]
		Assume first that $q>2$. Any facet distinct from $\mathcal B$ omits $B_a$, and $|B_a|=q-1\ge2$. Thus $\mathcal B$ shares no ridge with another facet. A facet distinct from $F_0$ is either $\mathcal B$ (omitting $A_0$) or contains some $A_i$ with $i\ge1$ (omitting $B_{a-1}$). Since $|A_0|=p^{a-1}(p-1)\ge2$ and $|B_{a-1}|=(p-1)(q-1)\ge2$, $F_0$ also shares no ridge with another facet. This contradicts Lemma~\ref{lem:ridge-obstruction}. 
		
		It remains to take $q=2$. If $a=2$, this gives $n=2p^2$. Suppose instead that $a\ge3$. Since $p\ne2$, we have $p-1\ge2$. Besides $F_0$, the clique complex has the facet $F_1=A_0\cup A_1\cup B_1\cup\cdots\cup B_{a-2}$. A facet distinct from $F_1$ omits either all of $B_{a-2}$ or all of $A_1$. Since $|B_{a-2}|=p(p-1)\ge2$ and $|A_1|=p^{a-2}(p-1)\ge2$, $F_1$ has no ridge-sharing partner. Moreover, a facet distinct from $F_0$ omits either $A_0$ or $B_{a-1}$, both of which have size at least 2. Thus $F_0$ has no ridge-sharing partner either, again contradicting Lemma~\ref{lem:ridge-obstruction}. The only surviving $p^a q$ cases are $2p$ and $2p^2$.
		
		\textbf{The case $n=pqr$:} 
		Use the Chinese remainder representation $\mathbb Z_{pqr}\cong\mathbb F_p\times\mathbb F_q\times\mathbb F_r$. Let $V_T$ be the set of vertices whose nonzero-coordinate support is $T \subsetneq \{1,2,3\}$. The facets of the independence complex are $F_i=V_{\{i\}}\cup V_{\{i,j\}}\cup V_{\{i,k\}}$ for $i \in \{1,2,3\}$ and $F_4=V_{\{1,2\}}\cup V_{\{1,3\}}\cup V_{\{2,3\}}$. The sizes are $|V_{\{1\}}|=p-1$, $|V_{\{2\}}|=q-1$, and $|V_{\{3\}}|=r-1$. At most one of $p,q,r$ is $2$, so at least two of these three cardinalities are at least two. For each corresponding $F_i$, its intersection with another star facet omits $V_{\{i\}}$ together with one two-element support class. Therefore at least two of $F_1,F_2,F_3$ share no ridge with any other facet, which contradicts Lemma~\ref{lem:ridge-obstruction}.
	\end{proof}
	
	\begin{example}
		We compute the Betti numbers of $\Gamma(\mathbb Z_{27})$ via Theorem~\ref{thm:scm_zdg_classification}. Since $27=3^3$, we have $p=3$ and $a=3$. The nonzero zero-divisors are
		\[
		V_2=\{9,18\}, \qquad V_1=\{3,6,12,15,21,24\}.
		\]
		Vertices in $V_2$ satisfy $2+2 \ge 3$, hence they form a clique $C$, while vertices in $V_1$ satisfy $1+1<3$ and form an independent set $U$. Moreover, $1+2 \ge 3$, so every vertex in $V_1$ is adjacent to every vertex in $V_2$. The resulting threshold graph structure is illustrated in Figure~\ref{fig:zdg_27}.
		
		\begin{figure}[htpb]
			\centering
			\begin{tikzpicture}[
				vertex/.style={circle, draw, fill=black, inner sep=0pt, minimum size=5pt},
				edge/.style={thick}
				]
				
				\draw[dashed, rounded corners=10pt, fill=gray!5] (-1.0, 1.5) rectangle (1.0, -1.5);
				\node[font=\normalsize\bfseries] at (0, 1.9) {$C = V_2$};
				
				\draw[dashed, rounded corners=10pt, fill=gray!5] (2.5, 3.5) rectangle (4.5, -3.5);
				\node[font=\normalsize\bfseries] at (3.5, 3.9) {$U = V_1$};
				
				\node[vertex, label=left:{$9$}] (v9) at (0, 0.5) {};
				\node[vertex, label=left:{$18$}] (v18) at (0, -0.5) {};
				\draw[edge] (v9) -- (v18);
				
				\node[vertex, label=right:{$3$}] (v3) at (3.5, 2.5) {};
				\node[vertex, label=right:{$6$}] (v6) at (3.5, 1.5) {};
				\node[vertex, label=right:{$12$}] (v12) at (3.5, 0.5) {};
				\node[vertex, label=right:{$15$}] (v15) at (3.5, -0.5) {};
				\node[vertex, label=right:{$21$}] (v21) at (3.5, -1.5) {};
				\node[vertex, label=right:{$24$}] (v24) at (3.5, -2.5) {};
				
				\foreach \c in {v9, v18} {
					\foreach \s in {v3, v6, v12, v15, v21, v24} {
						\draw[edge] (\c) -- (\s);
					}
				}
			\end{tikzpicture}
			\vspace{0.3cm} 
			\caption{The zero-divisor graph $\Gamma(\mathbb Z_{27})$.}
			\label{fig:zdg_27}
		\end{figure}
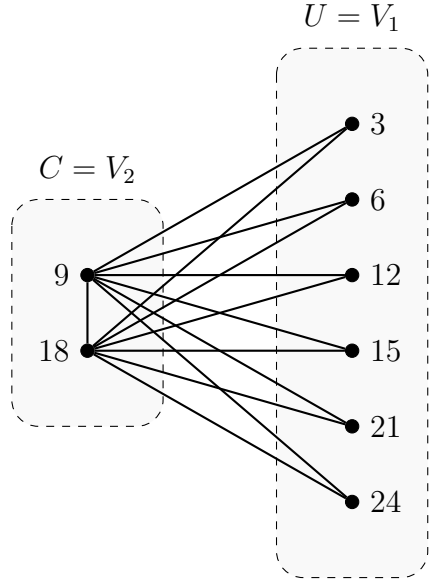
		
		Since $\left\lfloor \frac{a}{2}\right\rfloor =1$ and $\left\lceil \frac{a}{2}\right\rceil =2$, the summation in Theorem~\ref{thm:scm_zdg_classification} consists of a single valuation layer, namely $t=2$. Therefore,
		\[
		\beta_i(S/I(\Gamma(\mathbb Z_{27}))) = \sum_{t=2}^{2} 3^{3-t-1}(2) \binom{3^t - 2}{i} - \binom{3^1 - 1}{i+1}, \quad \text{for } i \ge1.
		\]
		Simplifying gives
		\[
		\beta_i(S/I(\Gamma(\mathbb Z_{27}))) = 2\binom{7}{i} - \binom{2}{i+1}, \quad \text{for } i \ge1.
		\]
		Since $\beta_0(S/I(\Gamma(\mathbb Z_{27})))=1$, evaluating these binomials yields the following Betti table for $S/I(\Gamma(\mathbb Z_{27}))$:
		\begin{center}
			\begin{tabular}{r|cccccccc} 
				& 0 & 1 & 2 & 3 & 4 & 5 & 6 & 7 \\
				\text{total:} & 1 & 13 & 42 & 70 & 70 & 42 & 14 & 2 \\\hline
				0 & 1 & $\cdot$ & $\cdot$ & $\cdot$ & $\cdot$ & $\cdot$ & $\cdot$ & $\cdot$ \\
				1 & $\cdot$ & 13 & 42 & 70 & 70 & 42 & 14 & 2
			\end{tabular}
		\end{center}
		In particular,
		\[
		\operatorname{pdim}(S/I(\Gamma(\mathbb Z_{27})))=7,
		\]
		since $\beta_7\neq0$ while $\beta_i=0$ for all $i>7$.
	\end{example}
	
	\begin{corollary}
Let $n\ge2$ and assume that $\Gamma(\mathbb Z_n)$ is co-chordal.  Then $S/I(\Gamma(\mathbb Z_n))$ is Cohen--Macaulay if and only if $n=p$ or $n=p^2$ for a prime $p$.
\end{corollary}
\begin{proof}
Cohen--Macaulayness implies sequential Cohen--Macaulayness, so Theorem~\ref{thm:scm_zdg_classification} reduces the possibilities to $p^a$, $2p$, and $2p^2$.  For $p^a$, Corollary~\ref{cor:chain_cm_exact} gives $a\le2$.  For $2p$ with $p$ odd, the graph is $K_{1,p-1}$ and its complement has maximal cliques of sizes $1$ and $p-1$, so it is not pure.  For $2p^2$, the split partition in the theorem has cross-degrees $p^2-p$ and $p-1$, neither of which satisfies the uniform Cohen--Macaulay alternatives in Corollary~\ref{cor:cm_split}.  Hence only $p$ and $p^2$ remain.
\end{proof}
	
	\section{Concluding Remarks}\label{sec:conclusion}

The ridge-compatible ordering of maximal cliques is fundamental to determining the Betti formula for sequentially Cohen--Macaulay co-chordal graphs. Once that order is fixed, the projective dimension, depth, and Cohen--Macaulay Betti sequence follow directly from the clique sizes.  The split-graph formula is obtained independently from this construction.  For nilpotent graphs of products of chain rings, co-chordality together with sequential Cohen--Macaulayness is equivalent to thresholdness and to the presence of at most one factor of nilpotency index at least three; the resulting valuation layers give closed binomial expressions.

For $n=p^a$, the coefficient of the valuation-$t$ contribution is
\[
p^{a-t-1}(p-1)=\varphi(p^{a-t}),
\]
where $\phi$ is Euler's totient function, so Theorem~\ref{thm:scm_zdg_classification} may be written as
\[
\beta_i(S/I(\Gamma(\mathbb Z_{p^a})))
=
\sum_{t=\lceil a/2\rceil}^{a-1}
\varphi(p^{a-t})\binom{p^t-2}{i}
-
\binom{p^{\lfloor a/2\rfloor}-1}{i+1}.
\]

This shows that the homological invariants of zero-divisor graphs over $\mathbb{Z}_{p^a}$ are determined by Euler's totient function, connecting the algebraic topology of graph ideals with classical number theory.
The classification for $\mathbb Z_n$ in this paper is stated within the co-chordal range.  Determining whether additional sequentially Cohen--Macaulay zero-divisor graphs occur outside that range requires methods that do not rely on a linear resolution and remains separate from the present argument.

\section*{Data Availability}
No datasets were generated or analyzed in this study.
	
	\section*{Competing Interests}
The author declares no competing interests.
	
	\section*{Acknowledgments}
	The author is grateful to Ralf Fr\"oberg for carefully reading an earlier manuscript, for his helpful correspondence, and for his valuable suggestions to improve this paper.
	
{\raggedright
	\bibliographystyle{amsplain}
\bibliography{SCM_Updated}
}
	
\end{document}